\documentclass[11pt]{amsart}
\usepackage{amsmath}
\usepackage{amsmath, amsthm, amscd, amsfonts, amssymb,color, tikz, enumerate}
\usepackage{amsfonts}
\usepackage{amsmath,amsthm,amscd,amssymb,mathrsfs,setspace}
\usepackage{latexsym,epsf,epsfig}
\usepackage{bm}
\usepackage{color}
\usepackage[hmargin=2.5cm,vmargin=3cm]{geometry}
\usepackage{cite}
\usepackage{mathtools}
\usepackage{cancel}
\usepackage[shortlabels]{enumitem}
\usepackage [english]{babel}
\usepackage [autostyle, english = american]{csquotes}
\usepackage{mathtools}
\MakeOuterQuote{"}

\newtheorem{theorem}{Theorem}[section]
\newtheorem{lemma}[theorem]{Lemma}
\newtheorem{proposition}[theorem]{Proposition}

\newtheorem{definition}{Definition}

\theoremstyle{remark}
\newtheorem{remark}{Remark}[section]
\numberwithin{equation}{section}

\theoremstyle{plain}

\newcommand{\comments}[1]{}

\newcommand{\Z}{\mathbb Z}

\DeclareMathOperator*{\esssup}{ess\,sup}

\newcommand{\D}{\displaystyle }
\newcommand{\dt}{{\D\frac{d}{dt}}}
\newcommand{\ra}{\rightarrow}
\newcommand{\lra}{\longrightarrow}
\newcommand{\be}{\begin{equation}}
\newcommand{\ee}{\end{equation}}
\newcommand{\bes}{\begin{equation*}}
\newcommand{\ees}{\end{equation*}}

\newcommand{\bea}{\begin{eqnarray}}
\newcommand{\eea}{\end{eqnarray}}
\newcommand{\beas}{\begin{eqnarray*}}
\newcommand{\eeas}{\end{eqnarray*}}

\def\mbf#1{\mathbf{#1}}

\def\<{\langle} 
\def\>{\rangle}

\title[Regularity criterion for 3D NSE based on finitely many observations]{A Novel Regularity Criterion For The three-dimensional  Navier-Stokes Equations Based On Finitely many observations}
\author{Abhishek Balakrishna}
\address{ (A. Balakhrishna) 
	Department of Mathematics and Statistics\\ 
	University of Maryland Baltimore County  \\ 
	Baltimore, MD 21250\\
	USA}
\email{bala2@umbc.edu}

\author{
	Animikh Biswas$^\dagger$}
\address{ (Animikh Biswas) 
	Department of Mathematics and Statistics\\ 
	University of Maryland Baltimore County  \\ 
	Baltimore, MD 21250\\
	USA}
\email{abiswas@umbc.edu}
\thanks{$^\dagger$ Corresponding author. Email: abiswas@umbc.edu}

\makeatletter
\newcommand*{\rom}[1]{\expandafter\@slowromancap\romannumeral #1@}
\makeatother

\begin{document}
 \maketitle
 
 \begin{abstract}
     In this paper we present two results: (1) A data assimilation algorithm for the 3D Navier-Stokes equation (3D NSE) using nodal data, and, as a consequence (2) a novel regularity criterion for 
     the 3D NSE based on finitely many observations of the velocity. The data assimilation algorithm we employ utilizes nudging, a method  based on a Newtonian relaxation scheme motivated by feedback-control. The observations, which may be either modal, nodal or volume elements,  are drawn from a weak solution of the 3D NSE and are collected \emph{almost everywhere} in time over a finite grid and our results, including the regularity criterion,  hold for data of any of the aforementioned forms. The regularity criterion we propose follows  from our data assimilation algorithm and is hence intimately connected to the notion of determining functionals (modes, nodes and volume elements). To the best of our knowledge, all existing regularity criteria require knowing the solution of the 3D NSE almost everywhere in space. Our regularity criterion is \emph{fundamentally different} from any preexisting regularity criterion as it is based on \emph{finitely many observations}  (modes, nodes and volume elements). We further prove that the regularity criterion we propose is both a \emph{necessary and sufficient condition for regularity}. Thus our result can be viewed as a natural generalization of the notion of determining modes, nodes and volume elements as well as the asymptotic tracking property of the nudging algorithm  for the 2D NSE to the 3D setting.
 \end{abstract}
 
    \section{Introduction}
    Consider the bounded domain $\Omega$, with boundary $\partial \Omega$, as a subset of $\mathbb{R}^3$. Assume $\Omega$ to be filled with a viscous incompressible fluid modeled by the 3D Navier-Stokes equation
    \begin{equation}\label{1}
\begin{split}
    \frac{\partial u}{\partial t}+(u\cdot\nabla)u-\nu\Delta u+\nabla p&=f\\
    \nabla\cdot u&=0\\
    u|_{t=0}&=u_0.
\end{split}
\end{equation}
 Foe simplicity, we take $\Omega=[0,L]^3$ and the boundary condition is taken to be periodic.
 Here, $u$ is the unknown fluid velocity, $p$ is the unknown pressure and $f$ is a given external force. For the purpose of this introduction, $f$ is taken to be smooth ($C^{\infty}(\Omega)$).
 
 The question of well-posedness for system \eqref{1} has been of great interest to mathematicians. In his seminal paper \cite{Leray} Leray proved the global existence of weak solutions for the system \eqref{1} for all divergence free initial data $u_0\in L^2(\Omega)$. Furthermore, weak solutions that obey the energy inequality \eqref{lerener} are called Leray-Hopf weak solutions. The uniqueness of Leray-Hopf weak solution was an open problem until recently. Results by Vicol and Buckmaster \cite{vicb} and Albritton, Brué and Colombo \cite{colombo} demonstrated the non-uniqueness of weak and Leray-Hopf weak solutions respectively. The question of whether all Leray-Hopf weak solutions with smooth initial data are \emph{regular} is still open to the best of our knowledge and is a millennium problem. A solution $u$ is said to be regular if $\|u(t)\|_{H^1}$ is uniformly bounded in time. As opposed to the case of weak solutions, the global existence of a regular solution is still an open problem. We only know certain partial results. For example, we have local in time existence of a regular solution with the time of existence depending on the "size" of $f$ and $u_0$. 
 
 Despite this lack of well-posedness in the classical sense, several regularity criteria have been found. Such results specify additional conditions under which a Leray-Hopf weak solution is regular. The first important regularity criterion is due to Serrin \cite{Serrin}: If $u$ is a weak solution that satisfies 
 \begin{equation}\label{lad}
     u\in L^s(0,T;L^q),~\text{ with }~~2<s<\infty, ~~3<q<\infty, ~~\frac{3}{q}+\frac{2}{s}<1,
 \end{equation}
then $u$ is regular on $(0,T)$. This result was further improved by showing that $\frac{3}{q}+\frac{2}{s}<1$ in \eqref{lad} can be replaced by the weaker condition $\frac{3}{q}+\frac{2}{s}\leq 1$. The equality yields the largest class, providing us with a more general regularity criterion : If $u$ is a weak solution that satisfies 
 \begin{equation}\label{lad1}
     u\in L^s(0,T;L^q),~\text{ with }~~2<s<\infty, ~~3<q<\infty, ~~\frac{3}{q}+\frac{2}{s}=1,
 \end{equation}
then $u$ is regular on $(0,T)$.
The borderline case of $q=3$ and $s=\infty$ is excluded from\eqref{lad1}. One of the first results in this direction was by von Wahl \cite{vW85}, who showed that the space $C(0, T ;L^3)$ of continuous $L^3$-valued functions is a regularity class. In \cite{Koz97}, it was shown that the space $BV (0, T ;L^3)$ of $L^3$-valued functions of bounded variation is also a regularity class. Further regularity classes within $L^{\infty}(0, T ;L^3)$ are given in \cite{Koz97}, \cite{dV97}. See \cite{Koz96} concerning a uniqueness result in this space.
Kozono \cite{Koz01} was able to improve the last result by replacing $L^3$ by a larger Lorentz space. This motivated mathematicians to look into more general Lorentz spaces \cite{bomi}\cite{koya98}. Expanding on this idea, Sohr \cite{sohr} was able to extend Serrin's result by introducing Lorentz spaces in both time and space. The space $L^{\infty}(0,T; L^3)$ was finally shown to be a regularity class by Escauriaza, Seregin and Sverak \cite{sverak}.

Another regularity class  was proposed by Da Veiga \cite{DaVeiga} : If a weak solution $u$ satisfies $$u\in L^s(0,T; W^{1,q})~\text{ with }~~1<s<\infty, ~~\frac{3}{2}<q<\infty, ~~\frac{3}{q}+\frac{2}{s}=2,$$
then it is regular on $(0,T)$. Here $W^{1,q}$ is the usual Sobolev space.  The  well-known Beale-Kato-Majda \cite{BKM} regularity criterion  asserts that if the vorticity $\omega = \nabla \times u \in L^1([0,T]; L^\infty(\Omega))$, then the solution $u$ is regular.
This immediately implies that  the broderline case $q=\infty, s=1$ of Daveiga's result above. 
Regularity results based on conditions on the pressure have also been established \cite{ber, chaelee, daveiga, kukavica, zhou}. Additionally,  regularity results based on only one component of the velocity for the 3D NSE on the whole space or with periodic boundary conditions have also been formulated \cite{he, kukaz, pokorny, titi, zhou02}.

The regularity criterion that we propose for the 3D NSE in this paper is \emph{fundamentally different} from the ones mentioned above and is based on \emph{finitely many observations}\cite{JT,AOT}. The observations that we consider are either
finitely many Fourier coefficients $\{\hat{u}(\mbf k,t), \mbf k=(k_1,k_2,k_3) \in \Z^3, |\mbf k|\le N\}$ \cite{fp}, or finitely many nodal values $\{u(x_i,t)\}_{i=1}^N$ where $\{x_i\}$ are points on an uniform grid covering the domain $\Omega$, or finitely many \emph{volume elements} \cite{JT}.
It is 
intimately linked to, and motivated by, the notion of determining functionals (modes, nodes and volueme elements) as well as \emph{data assimilation}. 
The concept of data assimilation arises in the context of forecasting using dynamical systems. While dealing with biological or physical systems, one is often hindered by a lack of adequate knowledge of the initial state and/or model parameters describing the system. In order to compensate for this, one may utilize available measurements of the system, collected  on a {\it much coarser (spatial) scale than the desired resolution of the forecast.} An example of this occurs in weather prediction where one has almost continuously collected data from sparsely located weather stations. The objective of data assimilation and signal synchronization is to use this coarse scale observational measurements to fine tune  our knowledge of the state and/or model to improve the accuracy of the forecasts \cite{Daley1991, Kalnay2003}.

Data assimilation algorithms can be broadly classified as probabilistic and deterministic. Many of the classical data assimilation algorithms were probabilistic methods based on linear quadractic estimation, also known as the Kalman filter. For nonlinear models, modified versions of the Kalman filter have been employed, such as the Ensemble Kalman Filter (EnKF), Extended Kalman Filter (EKF) and the Unscented Kalman Filter, (although, unlike the Kalman filter,  these do not enjoy the optimality property). Consequently, there has been a recent surge of interest in developing a rigorous mathematical framework for these approaches; see, for instance, \cite{asch,  HMbook2012,  Kalnay2003, KLS, LSZbook2015, ReichCotterbook2015} and the references therein. These works provide a Bayesian and variational framework for the problem, with emphasis on analyzing variational and Kalman filter based methods.

A predominantly used deterministic approach to data assimilation is the technique of \emph{nudging}, which utilizes a feedback control paradigm via the \emph{Newton relaxation scheme}. This method is often cheaper to employ and is based on the idea of existence of finitely many determining functionals (nodes, modes and volume averages) for a dissipative dynamical system. A rigorous mathematical analysis of this method in the context of fluid dynamics by incorporating observed data into the model via interpolation operators was first performed by Azouani, Olson and Titi \cite{AOT}. Due to this, the nudging algorithm will henceforth be referred to as the \emph{AOT system}.

A schematic description of the nudging algorithm is as follows.
Assuming that the observations are generated from a continuous dynamical system given by
\[
\dt u = F(u), u(0)=u_0,
\]
the associated AOT system is given by
\begin{equation} \label{aotsystem}
\dt w =F(w)-\mu I_h(w-u), w(0)=w_0\ (\mbox{arbitrary}).
\end{equation}
Here, $I_hu$ is an approximation of $u$ constructed solely from observational data. We mainly consider three types of observational data : modes (low fourier modes of $u$), nodes (pointwise measurements of $u$) and volume averages (local averages of $u$). The approximation is constructed through the use \emph{interpolation operators} $I_h$ acting on the phase space (See subsection \ref{interpolation}). $h$ refers to the size of the spatial grid over which the observed data was collected. In case of the \emph{modal interpolant}, the reciprocal of $h$ is proportional to the number of observed modes. Moreover,  $\mu>0$ is the \emph{relaxation/nudging parameter} an appropriate choice of which needs to be made for the algorithm to work, i.e. the system to be {\em globally (in time) well-posed and for its solution to possess the asymptotic  tracking property}, namely, $\|w-u\| \lra 0$ as $t \ra \infty$ in a suitable norm. This in turn critically depends on the richness of the observation space and bounds on the system variables.

One of the first results that \emph{hinted} a possible link between data assimilation and regularity for the 3D NSE was given in \cite{bisprice} where the authors  succeeded in applying the AOT data assimilation technique to the 3D NSE for the modal case, based on certain conditions on the observed modal data. The regularity of the data assimilated solution was crucial in demonstrating the tracking property which was in turn achieved under specific conditions on the observed modal data. There was hence a connection between the specific condition on the observed data and regularity. Now the question was whether we could extend the idea to obtain a condition on the observed data that would be sufficient to establish the regularity of the actual solution to the 3D NSE. We pursued this line of enquiry in \cite{BB} to establish an \emph{observable regularity criterion} for the 3D NSE on the \emph{weak attractor}, and based on observed data that were either modal or volume elements. The criterion was based on a condition on either the observed modes or the observed volume averages. It should be noted that the case of nodal observations was not covered by the technique in \cite{BB}.

In this paper, we \emph{substantially} generalize our result in \cite{BB} and propose a regularity criterion for a Leray-Hopf weak solution of the 3D NSE which not only covers the case of finitely many observed modes or volume elements, but also finitely many pointwise observations (nodes).
The new criterion, unlike the one in \cite{BB}, doesn't require the solution to be on the weak attractor, but is now formulated on a finite interval $[0,T]$. Due to the fact that our regularity criterion is based on finitely many observations, it is completely different from the classical ones described earlier. 

Before stating our main result, we define the quantity $M_{h,u}$, purely in terms of observed data as follows:
\begin{equation}\label{mh1}
       \begin{split}
        M_{h,u}^2=\esssup_{0\leq t< T}\begin{dcases}
        \displaystyle\|P_N(u)\|^2\sim \sum_{|k|\leq N}|\lambda_{k}|^2|\hat{u}(k)|^2,~~\frac{2\pi N}{L}\sim\frac{1}{h}&(\text{Modal})\\[10pt]
         \displaystyle Ch \sum_{\alpha}|\bar{u}_{\alpha}|^2,
          ~\bar{u}_\alpha = \frac{1}{|Q_\alpha|}\int_{Q_\alpha} u &(Volume)\\[10pt]
         \displaystyle Ch \sum_{\alpha}|u(x_\alpha)|^2,
           &(Nodal)
        \end{dcases}
        \end{split}
        \end{equation}
$M_h$ plays a key role in making an informed guess on the upper bound of $\|u
\|_{H^1}$. Our proposed regularity criterion is as follows:

 \begin{theorem}
           Let $u$ be a weak solution to the 3D NSE given by \eqref{3dnav} such that $u(0)\in V$ and $M_h$ be defined by \eqref{mh}.
           If there exists an $h>0$ such that
	    \begin{equation}
	       \max\left\{\nu\lambda_1,\frac{cW_{h}^4}{\nu^3},\frac{cW_h|f|}{\nu^2}\right\}\leq\frac{\nu}{4ch^2}, \quad\text{ where }\quad  W_{h}^2=\frac{c}{\nu^2\lambda_1} |f|^2 + M_{h}^2,
	    \end{equation}
	    then $u$ is regular and $\|u\|\leq W_h$.
        \end{theorem}
Here $V$ is the divergence free subspace of $H^1(\Omega)$ containing functions that obey the boundary conditions prescribed in \eqref{1}. To the best of our knowledge, this is the first regularity criterion for the 3D NSE that \emph{requires knowing the solution on only finitely many points}. Additionally, we were able to apply the AOT algorithm to the 3D NSE for nodal data. In fact, our regularity criterion follows as a natural consequence of the results we obtained for the data assimilation algorithm with nodal data. Furthermore, we were able to prove that the criterion we obtained is \emph{both a necessary and sufficient condition for regularity}.

The paper is structured as follows: \textbf{Section 2} introduces the notation, the interpolation operators and sets up the data assimilated system for the 3D NSE. \textbf{Section 3} addresses the well-posedness of the data assimilated 3D NSE and concludes with the proof of the tracking property (synchronization). \textbf{Section 4} states and discusses the regularity criterion. \textbf{Section 5} discusses the conditions under which a set of nodes are determining. Finally, the \textbf{Appendix} contains proofs of inequalities that have been crucially used to prove results throughout the paper.

    \section{Preliminaries}
    \subsection{Notation}	
    The 3-D incompressible Navier-Stokes equations (3D NSE) on a domain $\Omega=[0,L]^3\subset\mathbb{R}^3$ with time independent forcing (assumed for simplicity) is given by
\begin{equation}\label{3dnav1}
\begin{split}
    \frac{\partial u}{\partial t}+(u\cdot\nabla)u-\nu\Delta u+\nabla p&=f\\
    \nabla\cdot u&=0\\
    u(t=0)&=u_0.
\end{split}
\end{equation}
    Here $u$ denotes the velocity of the fluid, $p$ denotes the pressure, $\nu$ is the kinematic viscosity and $f$ is the body force.
    
    For simplicity, we assume periodic boundary conditions:
    
    $u$ is space periodic with period $L$ in all variables with space average zero, i.e., $\displaystyle\int_\Omega u=0.$
    
	Following \cite{cf, Temam}, we briefly introduce the notations. For $\alpha>0$, $H^{\alpha}(\Omega)$ is the usual $L^2$ based Sobolev space. We denote the inner product and norm of $L^2(\Omega)$ by $(\cdot,\cdot)$ and $|\cdot|$ respectively and the inner product and norm of $H^1(\Omega)$ by $((\cdot,\cdot))$ and $||\cdot||$ respectively.\\
	We define $\mathcal{V}$ as the space of $L$-periodic trigonometric polynomials from $\mathbb{R}^3$ to $\mathbb{R}^3$ that are divergence free and have zero average. $H$ denotes the closure of $\mathcal{V}$ in $\left(L^2(\Omega)\right)^3$ while $V$ denotes the closure of $\mathcal{V}$ in $\left(H^1(\Omega)\right)^3$.
	
	$H$ is endowed with the inner products
	$$(u,v)=\sum_{i=1}^3\int_{\Omega}u_i(x)v_i(x)dx$$ and the norm $|u|=(u,u)^{1/2}$. $V$ is endowed with the inner product
	$$((u,v))=\sum_{i,j=1}^3\int_{\Omega}\partial_ju_i(x)\partial_jv_i(x)dx$$
	and the associated norm $\|u\|=((u,u))^{1/2}$. We also denote by $P_{\sigma}$ the Leray-Hopf orthogonal projection operator from $L^2(\Omega)$ to $H$.
	
	\subsection{Interpolant Operators}\label{interpolation}
	As part of our data assimilation algorithm, we will be considering interpolation operators based on the spatial observations of the reference solution to the system \eqref{3dnav1}. We briefly introduce some important interpolation operators below. They can be broadly classified into two categories: Type-\rom{1} and Type-\rom{2}.\\
	A finite rank, bounded linear operator $I_h : \left(H^1(\Omega)\right)^3\to \left(L^2(\Omega)\right)^3$ is said to be a type-\rom{1} interpolant operator if there exists a dimensionless constant $c>0$
	such that
	\begin{equation}\label{intest}
	|I_h(v)|\leq c|v| \text{ and } |I_h(v)-v| \leq ch\|v\|~~ \forall v\in \left(H^1(\Omega)\right)^3.
	\end{equation}
	Two important examples of type-\rom{1} interpolation functions are:
	\begin{itemize}
	    \item \textbf{Modal interpolation}: In this case $I_h u=P_K(u)$ with $h\sim 1/\lambda_K^{1/2}$, where $P_K$ denotes the orthogonal projection onto the space spanned by the first $K$ eigenvectors of the Stokes operator $A$.  Indeed, one can easily check that it satisfies \eqref{intest}:
	\begin{equation}\label{modalest}
	|P_K(v)|\leq |v|~~\forall v\in L^2(\Omega) ~~\text{ and } ~~|P_K(v)-v| \lesssim\frac{1}{\lambda_K^{1/2}}\|v\|~~\forall v\in \left(H^1(\Omega)\right)^3.
	\end{equation}

	\item \textbf{Volume interpolation}: In this case, $\Omega$ is partitioned into $N$ smaller cuboids $Q_{\alpha}$, where $\alpha\in\mathcal{J}=\left\{(j, k, l)\in \mathbb{N}\times \mathbb{N}\times \mathbb{N}:1\leq j, k, l\leq \sqrt[3]{N}\right\}$. Each cuboid is of diameter  $h=\sqrt{3}L/\sqrt[3]{N}$. The interpolation operator is defined as follows:
	\begin{equation}\label{volint}
	\left(I_h(v)\right)_i = \sum_{\alpha\in\mathcal{J}}\bar{v}_{i,\alpha}\chi_{Q_{\alpha}}
	\end{equation}
	where 
	\begin{equation}\label{valpha}
	\bar{v}_{i,\alpha} = \frac{1}{|Q_\alpha|}\int_{Q_{\alpha}}v_i(x)dx,
	\end{equation}
	and $|Q_\alpha|$ is the volume of $Q_\alpha$. For $v\in V_0$, we define $I_hv=\left(I_h(v_1),I_h(v_2),I_h(v_3)\right),$ where $v=(v_1,v_2,v_3)$.
	\end{itemize}
	A linear operator $I_h : \left(H^2(\Omega)\right)^3\to \left(L^2(\Omega)\right)^3$ is said to be a type-\rom{2} interpolation operator if there exists dimensionless constants $c_1,$ $c_2$ such that
	such that
	\begin{equation}\label{intest2}
	|I_h(v)-v| \leq c_1h\|v\|+ c_2h^2|\Delta v|~~ \forall v\in \left(H^2(\Omega)\right)^3.
	\end{equation}
	\begin{itemize}
	    \item \textbf{Nodal Interpolation}: In this case the domain is split into subdomains similar to the volume interpolation case above and from each subdomain $Q_\alpha$, a representative point $x_\alpha$ is chosen. This represents taking measurements at a finite number of points in the domain. The function is given by
	    \begin{equation}\label{nodint}
	        Iv(\cdot)=I_h v(\cdot)=\sum_{j=1}^Nv(x_\alpha)\chi_{Q_\alpha}(\cdot),
	    \end{equation}
	    where $h=\sqrt{3}L/\sqrt[3]{N}$.
	\end{itemize}

	\subsection{Mathematical Model}
	\noindent
	In order to prescribe \eqref{3dnav1} as a precise PDE model, we introduce a few operators.
	Let $D(A)=V\cap \left(H^2(\Omega)\right)^3$ and $A : D(A)\to H$ be the unbounded linear operator defined by
	\begin{equation}\label{A}
	(Au,v)=((u,v)).
	\end{equation}
	We recall that $A$ is a positive self adjoint operator with a compact inverse. Moreover, there exists a complete orthonormal set of
	eigenfunctions $\phi_{j}\in H $, such that $A\phi_{j} = \lambda_{j}\phi_{j}$, where $0<\lambda_{1} \leq \lambda_{2} \leq \lambda_{3} \leq \dots $ are the
	eigenvalues of $A$ repeated according to multiplicity.
	
	\noindent
	We denote by $H_n$ the space spanned by the first $n$ eigenvectors of $A$ and the
	orthogonal projection from $H$ onto $H_n$ is denoted by $P_n$. We also have the
	Poincare inequality
	\begin{equation}\label{poincare}
	\lambda_1^{1/2}|v|\leq\|v\|, v\in V.
	\end{equation}
	
	\noindent
	Let $V'$ be the dual of $V$. We define the bilinear term $B:V\times V\to V'$ by
	\[
	\langle B(u,v),w\rangle_{V',V}=(((u\cdot\nabla) v),w)
	\]
		
	\noindent
	The bilinear term $B$ satisfies the orthogonality property
	\begin{equation}\label{orthoganal}
	B(u,w,w) = 0~ \forall~ u\in V,~w\in V.
	\end{equation}
	We recall some well-known bounds on the bilinear term for velocity in the 3D case.
	\begin{proposition}
		If $u,v\in V$ and $w\in H$, then
		\begin{equation}\label{nolinest1}
		|(B(u,v),w)|\leq c\|u\|_{L^6}\|\nabla v\|_{L^3}\|w\|_{L^2}\leq c\|u\|\|v\|^{1/2}|Av|^{1/2}|w|
		\end{equation}
		Moreover if $u,v,w\in V$, then
		\begin{equation}\label{nolinest2}
		|(B(u,v),w)|\leq c\|u\|_{L^4}\|\nabla v\|_{L^2}\|w\|_{L^4}\leq c|u|^{1/4}\|u\|^{3/4}\|v\||w|^{1/4}\|w\|^{3/4}
		\end{equation}
	\end{proposition}
	\noindent
We also recall the Ladyzhenskaya's inequality for three dimensions :
\begin{equation}\label{lady}
\|w\|_{L^4}\leq C|w|_0^{\frac{1}{4}}\|w\|_0^{\frac{3}{4}}
\end{equation}
\noindent
We denote by $P_{\sigma}$ the Leray-Hopf orthogonal projection operator from $L^2(\Omega)$ to $H$.
With the above notation, by applying $P_{\sigma}$ to \eqref{3dnav1}, we may express the 3-D Navier-Stokes equation in the following functional form:
\begin{equation}\label{3dnav}
\begin{split}
    \frac{\partial u}{\partial t}+B(u,u)+\nu A(u)&=f.\\
    \nabla\cdot u&=0\\
    u(t=0)&=u_0.
    \end{split}
\end{equation}
where, by abuse of notation, we denote $P_{\sigma}(f)$ by $f$.
The data assimilation algorithm is given by the solution $w$ of the equation
\begin{equation}\label{heq}
\begin{split}
     \frac{\partial w}{\partial t}+B(w,w)+\nu Aw&=f + \mu(\tilde{I}u-\tilde{I}w)\\
     \nabla\cdot w&=0\\
     w(t=0)&=0.
    \end{split}
\end{equation}
In the above equation, $w= w_h$ depends on $h$. For convenience of notation, we suppress the $h$ dependence of $w$. $\tilde{I}$ is a smoothed nodal interpolation operator as defined in \eqref{modef}.
The Galerkin approximation of \eqref{heq} is obtained by applying $P_n$ and is given by
\begin{equation}\label{heq1}
\begin{split}
     \frac{\partial w_n}{\partial t}+B(w_n,w_n)+\nu Aw_n&=P_n f + \mu P_n(\tilde{I}u-\tilde{I}w_n)\\
      \nabla\cdot w_n&=0\\
      w_n(t=0)&=0.
    \end{split}
\end{equation}

    \subsection{Well-Posedness}
    	\begin{definition}\label{definesol}
	    $u$ is said to be a weak solution to \eqref{3dnav1} if for all $T>0$,
	    \vspace{3mm}
	    \begin{itemize}
	        \item $\displaystyle u\in L^{\infty}(0,T;H)\cap L^{2}(0,T;V)$
	        \vspace{3mm}
	        \item $u$ satisfies, $\displaystyle\forall v\in V$,  a.e.t 
	    \begin{equation}
	    \begin{split}
	         &\frac{d}{dt}(u,v)+\nu((u,v))+(B(u,u),v)=(f,v)
	    \end{split}
	    \end{equation}
	    \end{itemize}
	    A Leray-Hopf weak solution additionally satisfies, a.e. $s$, and for all  $t \ge s$,  the energy inequality
	    \begin{align}
	         |u(t)|^2+\int_s^t\nu\|u(\sigma)\|^2d\sigma&\leq |u(s)|^2+\int_s^t(f(\sigma),u(\sigma))d\sigma.\label{lerener}
	    \end{align}
	    A weak solution is said to be a strong/regular solution if additionally $$u\in L^{\infty}(0,T; V)\cap L^2(0,T;D(A)).$$
	\end{definition}
	
    \begin{remark}
	From the above definition, we see that a weak solution to the 3D NSE, $u\in L^{\infty}(0,T;H)\cap L^{2}(0,T;V)$. Hence there exists a $M(u_0)\in \mathbb{R}$ such that
	\begin{equation}\label{unibound}
	|u|\leq M.
	\end{equation}
\end{remark}

    \section{Data Assimilation With Type-\rom{2} Interpolation Function}
    
    In this section we will be discussing the well-posedness and the regularity of our data assimilated system. Our end goal is to demonstrate that the data assimilated velocity \emph{asymptotically} approaches the actual velocity (synchronization).

    \subsection{Existence And Uniqueness Of Strong Solution}
    In order to discuss existence and uniqueness of strong solutions, we will have to impose conditions on our data and look at the term $\|Iu\|$.
For the case of modal interpolation, in addition to satisfying \eqref{intest}, $I_h$ also satisfies
\begin{equation}\label{mod}
\|I_hv\| \leq \|P_Nu\|\leq c\|v\|~~ \forall v\in H^1(\Omega),~N\sim\frac{1}{h}
\end{equation}
The piece-wise constant nodal interpolation operator (as in \eqref{nodint}) does not have enough regularity for us to bound the term $\|\tilde{I}u\|$. In order to establish an inequality similar to \eqref{mod} and also an inequality explicitly in terms of the data for nodal interpolation, we define a smoothed nodal interpolation operator $\tilde{I}$ as follows
\begin{equation}\label{modef}
\tilde{I}(v)(x) = \sum_{\alpha\in\mathcal{J}}v(x_\alpha)\phi_{\alpha}(x),~~ v\in H^1(\Omega) 
\end{equation}
where
\begin{equation}\label{phi}
\phi_{\alpha}=\rho_{\epsilon}*\psi_{Q_{\alpha}},~\rho_{\epsilon}(x)=\epsilon^{-3}\rho\left(x/{\epsilon}\right).
\end{equation}

\begin{equation}
    \psi_{Q_{\alpha}}(x) =\begin{cases}
\displaystyle \chi_{Q_{\alpha}} &,\text{ for }\alpha\not\in\mathcal{E}\\
\chi_{Q_{\epsilon,\alpha}} &,\text{ for }\alpha\in\mathcal{E}
\end{cases}
\end{equation}

\begin{equation}\label{rho}
\rho(x) =\begin{cases}
\displaystyle K_{0}\exp\left(\frac{-1}{1-x^2}\right) &,\text{ for }|x|<1\\
~0 &,\text{ otherwise}
\end{cases}
\end{equation}
and 
\begin{equation*}
(K_0)^{-1}=\int_{|x|<1}\exp\left(\frac{-1}{1-x^2}\right)dx.
\end{equation*}
We set $\epsilon=h/10.$
Repeating the calculations in the appendix of \cite{AOT}, we obtain that $\tilde{I}$ is a type-\rom{2} interpolation operator. Additionally, repeating the arguments in the appendix of \cite{BB} and plugging in $u(x_\alpha)$ for $\bar{v}_\alpha$ in Theorem 6.2, we may conclude that
\begin{equation}\label{modib}
    \|\tilde{I}u\|^2\leq Ch\sum_{\alpha\in\mathcal{J}}|u(x_\alpha)|^2.
\end{equation}
    Next, for a general Leray-Hopf weak solution $u$ of \eqref{3dnav}, we define the quantity $M_{h}$($=M_{h,u,T}$) as
    \begin{equation}\label{mh}
       \begin{split}
        M_{h,u}^2=\sup_{0\leq t\leq T}\begin{dcases}
        \displaystyle\|P_N(u)\|^2\sim \sum_{|k|\leq N}|\lambda_{k}|^2|\hat{u}(k)|^2,~N\sim\frac{1}{h}&(\text{Modal})\\[10pt]
         \displaystyle Ch \sum_{\alpha}|\bar{u}_{\alpha}|^2,
          ~\bar{u}_\alpha = \frac{1}{|Q_\alpha|}\int_{Q_\alpha} u &(Volume)\\[10pt]
         \displaystyle Ch \sum_{\alpha}|u(x_\alpha)|^2,
           &(Nodal)
        \end{dcases}
        \end{split}
        \end{equation}
    where, $\lambda_{k}$ is the $k^{th}$ smallest eigenvalue of the stokes operator $A$ corresponding to the eigenvector $\phi_{k}$ and $\displaystyle P_N(u)=\sum_{k=1}^N\hat{u}(k)\phi_{k}$.\\
    Combining \eqref{modib} and \eqref{mh}, we obtain the bound
    \begin{equation}
       \sup_{0\leq t\leq T} \|\tilde{I}u\|^2\leq cM_h^2.
    \end{equation}

\begin{theorem}\label{regularity}
           Let $u$ be as in \eqref{3dnav}, $w$ be as in \eqref{heq}, $\tilde{I}$ be the nodal interpolation function as in \eqref{modef} and $I_h$ be the volume interpolation function as in \eqref{volint}.
           Assume that there exists $h>0$ such that
	    \begin{equation}\label{condcond}
	        \max\left\{\nu\lambda_1,\frac{cW_h^4}{\nu^3}\right\}\leq\frac{\nu}{4ch^2},\quad \text{where}\quad W_h^2=\frac{c}{\nu^2\lambda_1} |f|^2 + M_h^2.
	    \end{equation}
	    Let $\mu$ be chosen such that
	    \begin{equation}\label{munodalcond1}
	         \max\left\{\nu\lambda_1,\frac{cW_h^4}{\nu^3}\right\}\leq\mu\leq\frac{\nu}{4ch^2}.
	    \end{equation}
	    Then $w$ is regular and $\|w\|\leq W_h.$ 
        \end{theorem}
        \begin{proof}
        We begin by establishing a priori estimates on the Galerkin system.
        Taking the inner product of \eqref{heq1} with $Aw_n$, we obtain
	    \begin{equation}\label{navest}
	    \begin{split}
	    \frac{1}{2}\frac{d}{dt}\|w_n\|^2 +\nu|Aw_n|^2=&-(B(w_n,w_n),Aw_n)+\mu(w_n-\tilde{I}w_n,Aw_n)\\
	    &-\mu\|w_n\|^2+\mu(\tilde{I}u,Aw_n)+(f,Aw_n)
	    \end{split}
	    \end{equation}
	    We bound each term below.
	    First, applying \eqref{nolinest1} and Young’s inequality, we have
	    \begin{align*}
	    |(B(w_n,w_n),Aw_n)|&\leq c\|w_n\|^{3/2}|Aw_n|^{3/2}\leq \frac{c}{\nu^3}\|w_n\|^6+\frac{\nu}{4}|Aw_n|^2.
    	\end{align*}
	    Next, from \eqref{intest}, \eqref{nodh}, Cauchy-Schwartz, Young’s inequality and the second inequality in \eqref{munodalcond1}, we have
    	\begin{align*}
	    |\mu(w_n-\tilde{I}w_n,Aw_n)|&\leq \mu|w_n-I_hw_n|\dot|Aw_n|+\mu|\tilde{I}w_n-I_hw_n| |Aw_n|\\
	    &\leq\frac{\mu^2 ch^2}{\nu}\|w_n\|^2+\frac{\mu^4ch^6}{\nu^3}\|w_n\|^2+\frac{\nu}{4}|Aw|^2\\
	    &\leq \frac{\mu}{4}\|w_n\|^2+\frac{\nu}{4}|Aw_n|^2
    	\end{align*}
    	Applying \eqref{A} and Cauchy-Schwarz, we obtain
    	\begin{align*}
    	\mu(\tilde{I}u,Aw_n)_0&\leq \mu\|\tilde{I}\|^2+\frac{\mu}{4}\|w_n\|^2.
    	\end{align*}
    	Lastly, applying Cauchy-Schwartz and Young’s inequality, we obtain
    	\begin{equation*}
    	    |(f,Aw_n)|\leq |f||Aw_n|\leq\frac{1}{\nu}|f|^2+\frac{\nu}{4}|Aw_n|^2
    	\end{equation*}
    	Inserting the above estimate into \eqref{navest}, we obtain
    	\begin{equation}\label{contra}
    	    \frac{d}{dt}\|w_n\|^2+\left(\mu-\frac{c}{\nu^3}\|w_n\|^4\right)\|w_n\|^2+\frac{\nu}{2}|Aw_n|^2\leq \frac{2}{\nu}|f|^2+2\mu\|\tilde{I}u\|^2
    	\end{equation}
    	Let $[0,T_1]$ be the maximal interval on which $\|w_n(t)\|\le W_h$ holds for $t\in[0,T_1]$ where $W_h$ as in \eqref{condcond}.
    Note that $T_1>0$ exists because we have $w(0)=0$. Assume  that $T_1 < T$. Then by continuity, we must have $\|w_n(T_1)\|=W_h$.
    Using 
    \eqref{munodalcond1}, for all $t \in [0,T_1]$, we obtain
    \begin{equation}\label{ned}
        \frac{d}{dt}\|w_n\|^2 + \frac{\mu}{2}\|w\|^2 +\frac{\nu}{2}|Aw_n|^2\le \frac{2}{\nu}|f|^2 +2\mu \|\tilde{I}u\|^2.
    \end{equation}
     Since $w_n(0)=0$, dropping the last term on the LHS and applying the Gronwall inequality we immediately obtain
    \begin{equation}\label{ex1}
    \|w_n\|^2 \le \frac{4}{\nu^2\lambda_1} |f|^2 + 4\sup_{s\in[0,T)} \|\tilde{I}u\|^2\le\frac{1}{2}W_h\ \forall t \in [0,T_1].
    \end{equation}
    This contradicts the fact that $\|w_n(T_1)\|=W_h$. Therefore we conclude that $T_1\ge T$
    and consequently, $\|w_n(t)\| \le W_h$ for all $t \in [0,T]$.\\
    Also, since $\|w_n(t)\| \le W_h$ for all $t \in [0,T]$, from \eqref{ned} we obtain that for $s\in[0,T-1]$
    \begin{equation}\label{awbound}
         \nu\int_s^{s+1}|Aw_n(t)|^2ds\leq \frac{4}{\nu}|f|^2 +\frac{1}{2}\mu M_h^2.
    \end{equation}
    The remainder of the proof of existence is similar to the proof of existence of strong solutions of the 2-D NSE.
    
    \vspace{3mm}
    
    We now address uniqueness.  Let $w_1$ and $w_2$ be two regular solutions of \eqref{heq} with $\mu$ satisfying \eqref{condcond}. Let $\tilde{w}=w_1-w_2.$ Since $w_1$ and $w_2$ are regular, so is $\tilde{w}$. $\tilde{w}$ satisfies
        \begin{equation}\label{uniq2}
	    \begin{split}
	\frac{d\tilde{w}}{dt}+\nu A\tilde{w}_n+B(\tilde{w},w_{1})+B(w_{2},\tilde{w})&= -\mu I_h\tilde{w}=\mu (\tilde{w}-I_h\tilde{w})-\mu\tilde{w}.
	\end{split}
	\end{equation}
	Taking the inner product of \eqref{uniq2} with $\tilde{w}$, we obtain
	\begin{equation}\label{mmmbop}
	\begin{split}
	\frac{1}{2}\frac{d}{dt}|\tilde{w}|^2+\nu\|\tilde{w}\|^2+\mu|\tilde{w}|^2&\leq |(B(\tilde{w},w_{1}),\tilde{w}|)|+|\mu(\tilde{w}-I_h(\tilde{w}))||\tilde{w}|.
	\end{split}
	\end{equation}
	We bound each term on the RHS.\\
	First, applying \eqref{nolinest2}, Cauchy-Schwartz, Young’s inequality and \eqref{nedeq}, we obtain
	\begin{equation*}
	\begin{split}
	|(B(\tilde{w},w_{1}),\tilde{w}|)|&\leq c|\tilde{w}|^{1/2}\|\tilde{w}\|^{3/2}\|w_{1}\|\leq \frac{c}{\nu^3}M_{h,i}^4|\tilde{w}|^2+\frac{\nu}{2}\|\tilde{w}_n\|^2.
	\end{split}
	\end{equation*}
	Using Cauchy-Schwartz, Young’s inequality and \eqref{nedeq}, we obtain
	\begin{equation*}
	\begin{split}
	\mu(\tilde{w}-I_h(\tilde{w}))||\tilde{w}|\leq\mu ch^2\|\tilde{w}\|^2+\frac{\mu|\tilde{w}|^2}{4}\leq \frac{\nu}{2}\|\tilde{w}\|^2+\frac{\mu|\tilde{w}|^2}{4}.
	\end{split}
	\end{equation*}
	Applying the above estimates and \eqref{nedeq}, we obtain
	\begin{equation}\label{mmmbop2}
	\begin{split}
	\frac{d}{dt}|\tilde{w}|^2+\frac{\mu}{2}|\tilde{w}|^2&\leq 0.
	\end{split}
	\end{equation}
	Now using Gronwall's inequality on the interval $[0,t]$ and using the fact that $\tilde{w}(0)=0$, we obtain
	\begin{equation}
	\begin{split}
	|\tilde{w}(t)|^2= 0.
	\end{split}
	\end{equation}
    We hence have uniqueness.
        
        \end{proof}
        
    \subsection{Synchronization}
    We now show that the data assimilated velocity asymptotically approaches the actual velocity.
    \begin{theorem}\label{Sync}
           Let $u$ be as in \eqref{3dnav} such that $u(0)\in V$, $M_h$ be defined by \eqref{mh} and $w$ be as in \eqref{heq}. Also, let $\tilde{w}=u-w.$
           Assume that there exists $h>0$ such that
	    \begin{equation}
	        \max\left\{\nu\lambda_1,\frac{cW_h^4}{\nu^3},\frac{cW_h|f|}{\nu^2}\right\}\leq\frac{\nu}{4ch^2}, \quad \text{where} \quad W_h^2=\frac{c}{\nu^2\lambda_1} |f|^2 + M_h^2.
	    \end{equation}
	    Let $\mu$ be chosen such that
	    \begin{equation}\label{munodalcond2}
	         \max\left\{\nu\lambda_1,\frac{cW_h^4}{\nu^3},\frac{cW_h|f|}{\nu^2}\right\}\leq\mu\leq\frac{\nu}{4ch^2}.
	    \end{equation}
	    Then $\tilde{w}$ is regular and $\|\tilde{w}(t)\|\leq \|u(0)\|e^{-\mu t/2}.$ In particular, if we have $T=\infty$, then $$\lim_{t\to \infty}\|\tilde{w}(t)\|=0.$$
        \end{theorem}
        \begin{proof}
        Note that since $u(0)\in V$, we have existence of strong solution up until some time $T_0$.
        Subtracting \eqref{3dnav1} from \eqref{heq1} and taking the inner product with $A\tilde{w}$, we obtain
	    \begin{equation}\label{navest2}
	    \begin{split}
	    \frac{1}{2}\frac{d}{dt}\|\tilde{w}\|^2 +\nu|A\tilde{w}|^2=&-(B(\tilde{w},w),Aw)-(B(\tilde{w},\tilde{w}),Aw)-(B(w,\tilde{w}),Aw)\\
	    &+\mu(\tilde{w}-\tilde{I}\tilde{w},A\tilde{w})-\mu\|\tilde{w}\|^2
	    \end{split}
	    \end{equation}
	    We bound each term below.
	    First, applying \eqref{nolinest1} and Young’s inequality, we have
	    \begin{align*}
	    |(B(\tilde{w},w),A\tilde{w})|&\leq c\|\tilde{w}\|\|w\|^{1/2}|Aw|^{1/2}|A\tilde{w}|\leq \frac{cW_h}{\nu}\|\tilde{w}\|^2|Aw|+\frac{\nu}{8}|A\tilde{w}|^2.\\
	    |(B(\tilde{w},\tilde{w}),A\tilde{w})|&\leq c\|\tilde{w}\|^{3/2}|A\tilde{w}|^{3/2}\leq \frac{c}{\nu^3}\|\tilde{w}\|^6+\frac{\nu}{8}|A\tilde{w}|^2.\\
	    |(B(w,\tilde{w}),A\tilde{w})|&\leq c\|w\|\|\tilde{w}\|^{1/2}|A\tilde{w}|^{3/2}\leq \frac{cW_h^4}{\nu^3}\|\tilde{w}\|^2+\frac{\nu}{8}|A\tilde{w}|^2.
    	\end{align*}
	    Next, from \eqref{intest}, \eqref{nodh}, Cauchy-Schwartz, Young’s inequality and the second inequality in \eqref{munodalcond1}, we have
    	\begin{align*}
	    |\mu(\tilde{w}-\tilde{I}\tilde{w},A\tilde{w})|&\leq \mu|\tilde{w}-I_h\tilde{w}|\dot|A\tilde{w}|+\mu|\tilde{I}\tilde{w}-I_h\tilde{w}|\dot|A\tilde{w}|\\
	    &\leq\frac{\mu^2 ch^2}{\nu}\|\tilde{w}\|^2+\frac{\mu^4ch^6}{\nu^3}\|\tilde{w}\|^2+\frac{\nu}{4}|A\tilde{w}|^2\\
	    &\leq \frac{\mu}{4}\|\tilde{w}\|^2+\frac{\nu}{4}|A\tilde{w}|^2
    	\end{align*}
    	Inserting the above estimate into \eqref{navest2}, we obtain
    	\begin{equation}
    	    \frac{d}{dt}\|\tilde{w}\|^2+\left(\mu-\frac{c}{\nu^3}\|\tilde{w}\|^4-\frac{cW_h}{\nu}|Aw|\right)\|\tilde{w}\|^2+\nu |A\tilde{w}|^2\leq 0
    	\end{equation}
    	Let $[0,T_1]$ be the maximal interval on which $\|\tilde{w}(t)\|\le 2\|u(0)\|$ holds for $t\in[0,T_1]$.
    Note that $T_1>0$ exists because we have $\tilde{w}(0)=\|u(0)\|\in V$. Assume  that $T_1 < T$. Then by continuity, we must have $\|\tilde{w}(T_1)\|=2\|u(0)\|$.
    Using \eqref{munodalcond2}, for all $t \in [0,T_1]$, we obtain
    \begin{equation}
        \frac{d}{dt}\|\tilde{w}\|^2 + \left(\frac{\mu}{2}-\frac{cW_h}{\nu}|Aw|\right)\|\tilde{w}\|^2 \le 0.
    \end{equation}
     Applying the Gronwall inequality, \eqref{awbound} and \eqref{munodalcond2}, we immediately obtain
    \begin{equation}\label{exx1}
    \|\tilde{w}(t)\|^2 \le \|\tilde{w}(0)\|e^{-t\left(\mu-\frac{cW_h^2\sqrt{\mu}}{\nu}-\frac{cW_h|f|}{\nu^2}\right)}\le\|\tilde{w}(0)\|e^{-\mu t/2}\le\|u(0)\| ~~\forall t \in [0,T_1].
    \end{equation}
    This contradicts the fact that $\|w(T_1)\|=2\|u(0)\|$. Therefore we conclude that $T_1\ge T$
    and consequently, $\|\tilde{w}(t)\| \le 2\|u(0)\|$ for all $t \in [0,T]$.
        \end{proof}
    
    \section{Regularity Criterion For The 3D NSE}
    From Theorem \ref{regularity} we obtain regularity of $w$ while Theorem \ref{Sync} provides us with the regularity of $\tilde{w}=u-w$. Both of these results are obtained when the data obtained for the mesh size $h$ satisfies certain conditions. Using the triangle inequality $\|u\|\leq\|w-u\|+\|w\|$, we can combine the results to obtain a regularity criterion for the solution $u$ of the 3D NSE, which is stated below:
    
    \begin{theorem}\label{regcond}
           Let $u$ be a weak solution to the 3D NSE given by \eqref{3dnav} such that $u(0)\in V$. Let $M_h$ be defined by \eqref{mh}.
           If there exists an $h>0$ such that
	    \begin{equation}\label{regcrit}
	       \max\left\{\nu\lambda_1,\frac{cW_{h}^4}{\nu^3},\frac{cW_h|f|}{\nu^2}\right\}\leq\frac{\nu}{4ch^2}, \quad\text{ where }\quad  W_{h}^2=\frac{c}{\nu^2\lambda_1} |f|^2 + M_{h}^2,
	    \end{equation}
	    then $u$ is regular and $\|u\|\leq W_h$.
        \end{theorem}
        
    \begin{remark}
        The exact same proof strategy for Theorem \ref{regularity} and Theorem \ref{Sync} can be used to obtain the regularity of $w$ and $\tilde{w}$ when $I_h$ is a type-\rom{1} interpolation operator. Hence, we can obtain an analogous regularity criterion for the solution $u$ of the 3D NSE in the type-\rom{1} interpolation operator case.
    \end{remark}
    We have shown that if $u$ satisfies the criterion given in Theorem \ref{regcond}, then it is regular. Now the question arises : what about the converse? Will a regular solution to the 3D NSE satisfy the regularity criterion? The answer, as it turns out, is yes. We prove this in the theorem below.
    We now proceed to prove the theorem.
    \begin{theorem}\label{suff}
          Let $u$ be a regular solution to the 3D NSE  by \eqref{3dnav}and $M_h$ be defined by \eqref{mh}. Then there exists an $h>0$ such that $u$ satisfies \eqref{regcrit}.
    \end{theorem}
    \begin{proof}
        Looking at \eqref{regcrit}, we may say that since $\nu\lambda_1$ and $|f|$ are fixed and bounded for a given problem, it is sufficient to show that regularity of $u$ implies that $M_h$ is finite. Then the criterion can be satisfied by choosing $h$ sufficiently small.
        Recall that our domain is divided into smaller cubes $Q_\alpha$, indexed by $\alpha$. Also, each cube $Q_\alpha$ has side length $h$.\\
        First, using Proposition \ref{diffi}, Holder's inequality and Gagliardo-Nirenberg-Sobolev inequality, we obtain
        \begin{equation*}
            \begin{split}
                \sum_{\alpha}|u(x_\alpha)|^2=&\sum_{\alpha}\left(\frac{1}{|Q_\alpha|}\int_{Q_\alpha}u(x_\alpha)dx\right)^2\\
                \leq&~\sum_{\alpha}\left(\frac{1}{|Q_\alpha|}\int_{Q_\alpha}|u(x_\alpha)-u(x)|dx\right)^2+\sum_{\alpha}\left(\frac{1}{|Q_\alpha|}\int_{Q_\alpha}|u(x)|dx\right)^2\\
                &+2\sum_{\alpha}\left(\frac{1}{|Q_\alpha|}\int_{Q_\alpha}|u(x_\alpha)-u(x)|dx\right)\left(\frac{1}{|Q_\alpha|}\int_{Q_\alpha}|u(x)|dx\right)\\
                \leq&~ \frac{1}{h^3}\sum_{\alpha}\int_{Q_\alpha}|u(x_\alpha)-u(x)|^2dx+\frac{1}{h}\sum_{\alpha}\|u\|^2_{H^1(Q_\alpha)}\\
                &+\sum_{\alpha}\frac{2}{h^6}\left(\int_{Q_\alpha}|u(x_\alpha)-u(x)|^2dx\right)^{1/2}\sum_{\alpha}\left(\int_{Q_\alpha}|u(x)|^2dx\right)^{1/2}\\
                \leq&~ c\|u\||Au|+\frac{1}{h}\|u\|^2
            \end{split}
        \end{equation*}
        From \eqref{mh} and the above estimate, we have
        \begin{equation*}
        \begin{split}
            M_h^2=ch\sup_{t\in[0,T]}\sum_{\alpha}|u(x_\alpha)|^2&\leq c\left(h\sup_{t\in[0,T]}\|u(t)\||Au(t)|+\sup_{t\in[0,T]}\|u(t)\|\right)
        \end{split}
        \end{equation*}
        Since $u$ is regular, $\displaystyle\sup_{t\in[0,T]}\|u(t)\|$ is bounded. Also, from \cite{fmrt}, we have that $\displaystyle\sup_{t\in[0,T]}|Au(t)|$ is bounded. Hence $M_h$ is finite and its value decreases with $h$. Hence, by choosing a small enough $h$, condition \eqref{regcrit} can be satisfied.
    \end{proof}
    \begin{remark}
        Theorem \ref{regcond} gives us a regularity criterion for the 3D NSE based on \emph{finitely observed data}. All the other existing regularity criteria for the 3D NSE, to the best of our knowledge, require knowing $u$ or $p$ almost everywhere in the domain.
        
        Additionally, Theorem \ref{regcond} and Theorem \ref{suff} together prove that condition \eqref{regcrit} is \emph{both necessary and sufficient for the regularity of} $u$.
    \end{remark}
    \section{Determining Nodes}
    Lastly, we look into the \emph{asymptotically determining} aspect of nodal interpolation. Using arguments similar to the uniqueness part of Theorem \ref{regularity}, we observe that if $h$ satisfies \eqref{regcrit}, then $h$ is \emph{asymptotically determining}. 
    
     \begin{theorem}\label{asympdet}
           Let $u_1$ and $u_2$ be two weak solutions of \eqref{3dnav} and let $M_{h,i}$ be defined by \eqref{mh} for $u_i$, $i=1,2$. Assume that there exists $h>0$ such that
	    \begin{equation}\label{nedeq}
	        \max\left\{\nu\lambda_1,\frac{cW_{h,i}^4}{\nu^3},\frac{cW_{h,i}|f|}{\nu^2}\right\}\leq\frac{\nu}{4ch^2}, \quad\text{ where }\quad  W_{h,i}^2=\frac{c}{\nu^2\lambda_1} |f|^2 + M_{h,i}^2.
	    \end{equation}
	   If
	    \begin{equation}\label{inq}
	        \lim_{t\to\infty}|\tilde{I}_h(u_1-u_2)|=0,
	    \end{equation}
	    then
	    \begin{equation*}
	    \lim_{t\to\infty}|u_1-u_2|=0.
	    \end{equation*}
        \end{theorem}
    \begin{proof}
        Let $w_1$ and $w_2$ be the solution for the data assimilated equation \eqref{heq} corresponding to $u_1$ and $u_2$ for $\mu$ satisfying \eqref{munodalcond2} for both $u_1$ and $u_2$. From Theorem \ref{Sync} and \eqref{poincare}, we obtain that 
        \begin{equation}\label{lem}
            \lim_{t\to\infty}|u_i(t)-w_i(t)|^2\leq\lim_{t\to\infty}\frac{1}{\lambda_1}\|u_i(t)-w_i(t)\|^2=0.
        \end{equation}
        Now, let $\tilde{w}=w_1-w_2.$ Then, from \eqref{lem}, it is sufficient to show that $\displaystyle \lim_{t\to\infty}|\tilde{w}(t)|=0.$ $\tilde{w}$ satisfies
        \begin{equation}\label{uniq3}
	    \begin{split}
	\frac{d\tilde{w}}{dt}+\nu A\tilde{w}_n+B(\tilde{w},w_{1})+B(w_{2},\tilde{w})&=\mu (\tilde{I}_h(u_1-u_2)-I_h(\tilde{w}))
	\end{split}
	\end{equation}
	Taking the inner product of \eqref{uniq3} with $\tilde{w}$, we obtain
	\begin{equation}\label{mmmbop3}
	\begin{split}
	\frac{1}{2}\frac{d}{dt}|\tilde{w}|^2+\nu\|\tilde{w}\|^2+\mu|\tilde{w}|^2&\leq |(B(\tilde{w},w_{1}),\tilde{w}|)|+|\mu(\tilde{w}-I_h(\tilde{w}))||\tilde{w}|+\mu|\tilde{I}_h(u_1-u_2)||\tilde{w}|.
	\end{split}
	\end{equation}
	Bounding each term as in the the proof of uniqueness for Theorem \ref{regularity} and applying \eqref{nedeq}, we obtain
	\begin{equation}\label{mmmbop4}
	\begin{split}
	\frac{d}{dt}|\tilde{w}|^2+\frac{\mu}{2}|\tilde{w}|^2&\leq 2\mu|\tilde{I}_h(u_1-u_2)|.
	\end{split}
	\end{equation}
	Now using Gronwall's inequality on the interval $[\sigma,t]$, we obtain
	\begin{equation}
	\begin{split}
	|\tilde{w}(t)|^2&\leq |\tilde{w}(\sigma)|^2e^{-\frac{\mu}{2}(\sigma-t)}+4\sup_{t\in[\sigma,t]}|\tilde{I}_h(u_1-u_2)|.
	\end{split}
	\end{equation}
	First letting $t\to\infty$, and then applying \eqref{inq} while letting $\sigma\to\infty$ , we obtain
	\begin{equation}\label{fi}
	   \lim_{t\to\infty}| \tilde{w}(t)|^2= 0
	\end{equation}
        \end{proof}
        
        \begin{remark}
            In simpler words, the above theorem states that if you choose to observe two velocity function at sufficiently large number of points and they approach each other asymptotically at these points, then they approach each other asymptotically everywhere. This means if you construct an approximation of the actual function by observing it at sufficiently many points, the approximation must asymptotically approaches the actual function. This is the central idea of data assimilation.
        \end{remark}

\section{Appendix}
    We utilize this section to establish  a few key inequalities regarding the nodal interpolation operator. 
    
    In order to prove the next proposition, we first state and prove the following lemma which is a modification of Proposition A.1. in \cite{AFMTDown}.
    
    \begin{lemma}\label{modify}
        Let $0<l$ and $Q$ be a cube $[0,l]^3\subset\mathbb{R}^3.$ Then for $\phi\in H^2(Q)$ and $(x_1,y_1,z_1),(x_2,y_2,z_2)\in Q$, we have
        \begin{equation*}
            |\phi(x_1,y_1,z_1)-\phi(x_2,y_2,z_2)|\leq C\|\nabla \phi\|_{L^2(Q)}^{1/2}\|A\phi\|_{L^2(Q)}^{1/2}.
        \end{equation*}
    \end{lemma}
    \begin{proof}
        First consider $\psi=\psi(x,y,z)\in C^\infty(Q)$ and let $\tilde{y},\tilde{z}\in[0,l]$ and $\tilde{z}\in[0,l]$. Without loss of generality, assume $\tilde{y}$ and $\tilde{z}$ are closer to $0$ than $l$, i.e., 
        \begin{equation}\label{tildaineq}
            \tilde{y}\leq l-\tilde{y}~\text{ and }~\tilde{z}\leq l-\tilde{z}.
        \end{equation}
        Then, for every $x\in [0,l]$, $y\in[\tilde{y},h]$ and $z\in[\tilde{z},h]$, we have
        \begin{equation}
            \psi^2(x,\tilde{y},\tilde{z})=\psi^2(x,y,z)-\int_{\tilde{y}}^y\frac{\partial \psi^2(x,s,z)}{\partial y}ds-\int_{\tilde{z}}^z\frac{\partial \psi^2(x,\tilde{y},s)}{\partial z}ds
        \end{equation}
        Integrating with respect to $x$, $y$ and $z$ over $[0,l]\times [\tilde{y},l]\times[\tilde{z},l]$, we obtain
        \begin{equation}
            (l-\tilde{y})(l-\tilde{z})\int_0^l\psi^2(x,\tilde{y},\tilde{z})dx\leq \|\psi\|_{L^2(Q)}^2+2(l-\tilde{y})\|\psi\|_{L^2(Q)}\left\|\frac{\partial\psi}{\partial y}\right\|_{L^2(Q)}^2+2(l-\tilde{z})\|\psi\|_{L^2(Q)}\left\|\frac{\partial\psi}{\partial z}\right\|_{L^2(Q)}^2.
        \end{equation}
        Using \eqref{tildaineq}, we observe that 
        \begin{equation*}
            l-\tilde{y}\geq l/2>0 ~\text{ and }~l-\tilde{z}\geq l/2>0.
        \end{equation*}
        Therefore
        \begin{equation}\label{needint}
            \int_0^l\psi^2(x,\tilde{y},\tilde{z})dx\leq \frac{4}{l^2}\|\psi\|_{L^2(Q)}^2+\frac{4}{l}\|\psi\|_{L^2(Q)}\left\|\frac{\partial\psi}{\partial y}\right\|_{L^2(Q)}^2+\frac{4}{l}\|\psi\|_{L^2(Q)}\left\|\frac{\partial\psi}{\partial z}\right\|_{L^2(Q)}^2.
        \end{equation}
        For the case $l-\tilde{y}<l$, we choose $y\in [0,\tilde{y}]$ instead and proceed analogously. Similarly, for $l-\tilde{z}<l$, we choose $z\in [0,\tilde{z}]$. Also, since $\psi\in C^\infty(Q)$, by density, \eqref{needint} is also valid for every $\psi\in H^1(Q).$
        Now, let $\phi\in C^\infty(Q)$ and $(x_1,y_1,z_1),(x_2,y_2,z_2)\in Q$. By triangle inequality,
    \begin{equation}\label{q1}
    \begin{split}
        |\phi(x_1,y_1,z_1)-\phi(x_2,y_2,z_2)|\leq |\phi(x_1,y_1,z_1)-\phi(x_1,y_1,z_2)|&+|\phi(x_1,y_1,z_2)-\phi(x_1,y_2,z_2)|\\
        &+|\phi(x_1,y_2,z_2)-\phi(x_2,y_2,z_2)|.
    \end{split}
    \end{equation}
    Note that 
    \begin{equation}\label{hold}
        |\phi(x_1,y_1,z_1)-\phi(x_1,y_1,z_2)|=\left|\int_{z_1}^{z_2}\frac{\partial\phi}{\partial z}ds\right|\leq l^{1/2}\left\|\frac{\partial\phi}{\partial z}\right\|_{L^2(Q)}.
    \end{equation}
    Hence applying \eqref{needint} to \eqref{hold} with $\psi=\partial\phi/\partial z$, we obtain
    \begin{equation}\label{q2}
        |\phi(x_1,y_1,z_1)-\phi(x_2,y_2,z_2)|\leq \left(\frac{4}{l^2}\|\nabla\phi\|_{L^2(Q)}^2+\frac{8}{l}\|\nabla\phi\|_{L^2(Q)}\left\|A\phi\right\|_{L^2(Q)}^2\right)^{1/2}
    \end{equation}
    Similarly, we obtain
    \begin{equation}\label{q3}
        \begin{split}
            |\phi(x_1,y_1,z_2)-\phi(x_2,y_2,z_2)|&\leq \left(\frac{4}{l^2}\|\nabla\phi\|_{L^2(Q)}^2+\frac{8}{l}\|\nabla\phi\|_{L^2(Q)}\left\|A\phi\right\|_{L^2(Q)}^2\right)^{1/2}\\
            |\phi(x_1,y_2,z_2)-\phi(x_2,y_2,z_2)|&\leq \left(\frac{4}{l^2}\|\nabla\phi\|_{L^2(Q)}^2+\frac{8}{l}\|\nabla\phi\|_{L^2(Q)}\left\|A\phi\right\|_{L^2(Q)}^2\right)^{1/2}
        \end{split}
    \end{equation}
    Combining \eqref{q1}, \eqref{q2} and \eqref{q3} and using the density of $C^\infty(Q)$ in $H^2(Q)$, we have 
    \begin{equation*}
            |\phi(x_1,y_1,z_1)-\phi(x_2,y_2,z_2)|\leq \left(\frac{C_1}{l}\|\nabla \phi\|_{L^2(Q)}^2+C_2\|\nabla \phi\|_{L^2(Q)}\|A\phi\|_{L^2(Q)}\right)^{1/2}.
        \end{equation*}
        Now applying the fact that $\|\nabla \phi\|_{L^2(Q)}\leq Cl\|A \phi\|_{L^2(Q)}$, we obtain the statement of the theorem.
    \end{proof}
    \begin{proposition}\label{diffi}
        Let $I$ be a type-\rom{2} interpolation operator and $I_h$ be as in \eqref{volint}. Then 
        \begin{equation}\label{nodh}
            \|\tilde{I}u-I_hu\|_{L^2(\Omega)}^2\leq Ch^3\|u\||Au|.
        \end{equation}
        \end{proposition}
        \begin{proof}
            Using \eqref{nodint}, we note that
            \begin{equation}\label{difi}
                \begin{split}
                     \tilde{I}u-I_hu&=\sum_{\alpha\in\mathcal{J}}\left(u(x_\alpha)-\frac{1}{|Q_{\alpha}|}\int_{Q_{\alpha}}u(x)dx\right)\chi_{Q_{\alpha}}\\
                &=\sum_{\alpha\in\mathcal{J}}\left(\frac{1}{|Q_{\alpha}|}\int_{Q_{\alpha}}\left(u(x_\alpha)-u(x)\right)dx\right)\chi_{Q_{\alpha}}.
                \end{split}
            \end{equation}
            Now, using \eqref{difi}, Lemma \eqref{modify} and the fact that 
            \begin{equation}
            	\chi_{Q_{\alpha}}\chi_{Q_{\beta}}=
            	\begin{cases} 
            	0 &\mbox{if } \alpha\neq\beta \\
            	\chi_{Q_{\alpha}} & \mbox{if } \alpha=\beta ,
        	\end{cases} 
	\end{equation}
	we obtain
        \begin{equation*}
            \begin{split}
             \|\tilde{I}u-I_hu\|_{L^2(\Omega)}^2\leq& \left\langle \tilde{I}u-I_hu,\tilde{I}u-I_hu\right\rangle\\
             \leq&\sum_{\alpha\in\mathcal{J}}\int_{Q_\alpha} \frac{1}{|Q_{\alpha}|^2}\left(\int_{Q_{\alpha}}\left(u(x_\alpha)-u(s)\right)ds\right)^2dx\\
             \leq&\sum_{\alpha\in\mathcal{J}}\int_{Q_\alpha} \frac{|Q_{\alpha}|}{|Q_{\alpha}|^2}\left(\int_{Q_{\alpha}}\left|u(x_\alpha)-u(s)\right|^2ds\right)dx\\
             \leq&\sum_{\alpha\in\mathcal{J}}\left|Q_\alpha\right|\left(C\|\nabla u\|_{L^2(Q_{\alpha})}\|Au\|_{L^2(Q_{\alpha})}\right)\\
             \leq&~Ch^3\sum_{\alpha\in\mathcal{J}}\|\nabla u\|_{L^2(Q_{\alpha})}\|Au\|_{L^2(Q_{\alpha})}\\
             \leq &~Ch^3\left(\sum_{\alpha\in\mathcal{J}}\|\nabla u\|_{L^2(Q_{\alpha})}^2\right)^{1/2}\left(\sum_{\alpha\in\mathcal{J}}\|Au\|_{L^2(Q_{\alpha})}^2\right)^{1/2}\\
             \leq&~Ch^3\|u\||Au|
        \end{split}
        \end{equation*}

        \end{proof}

\end{document}